\documentclass{amsart}
\usepackage{amsthm, amsmath,textcomp,gensymb,amssymb,color,graphics,amstext,amscd, txfonts, epsfig, psfrag, color, multicol, graphicx}
\usepackage[normalem]{ulem}
\usepackage[colorlinks=true, pdfstartview=FitV, linkcolor=blue, citecolor=blue, urlcolor=blue]{hyperref}

\newtheorem{thm}{Theorem}
\newtheorem{cor}[thm]{Corollary}
\newtheorem{lemma}[thm]{Lemma}
\newtheorem{prop}[thm]{Proposition}

\theoremstyle{definition}
\newtheorem{remark}[thm]{Remark}
\newtheorem*{Acknowledgements}{Acknowledgements}

\definecolor{magenta}{rgb}{.5,0,.5} 
\definecolor{dred}{rgb}{.5,0,0} 
\definecolor{green}{rgb}{0,.5,0} 
\definecolor{blue}{rgb}{0,0,0.5} 
\definecolor{black}{rgb}{0,0,0} 
\definecolor{dgreen}{rgb}{0,.3,0} 
\definecolor{vdred}{rgb}{.3,0,0} 
\definecolor{red}{rgb}{1,0,0} 
\definecolor{gray}{rgb}{.5,.5,.5} 
\definecolor{cerulean}{rgb}{0,.48,.65} 
\definecolor{gold}{rgb}{0.80,0.58,0.05}

\def\ni{\noindent}

\def\CAT{\hbox{\rm CAT}}

\def\ms{\medskip}

\def\onto{{\kern3pt\to\kern-8pt\to\kern3pt}}
\def\<{\langle}
\def\>{\rangle}
\def\|{{\ |\ }}

\def\*{^{\star}}

\begin{document}

\title{Cannon--Thurston maps do not always exist}

\author{O.\ Baker  and T.R.\ Riley}

\thanks{The second author gratefully acknowledges partial  support from NSF grant DMS--1101651 and from Simons Foundation Collaboration Grant 208567.}  

\date \today

\begin{abstract}
\ni  We construct a hyperbolic group with a  hyperbolic subgroup for which inclusion does not induce a continuous map of the boundaries.

 \ms

\footnotesize{\ni \textbf{2010 Mathematics Subject
Classification:  20F67}  \\ \ni \emph{Key words and phrases:} hyperbolic group, Cannon--Thurston map}
\end{abstract}

\maketitle

\section{Introduction} \label{intro}
\emph{Hyperbolic groups} are the finitely generated groups whose Cayley graphs display characteristics of negative curvature.  Their systematic study was initiated by Gromov in \cite{Gromov4} and, mimicking the study of Riemannian manifolds, pays particular attention to large--scale and asymptotic geometric features such as boundaries. 

One of the many equivalent definitions of the \emph{Gromov boundary} $\partial G$ of an infinite hyperbolic group $G$ with word metric $d$ is as the set of equivalence classes of sequences $(a_n)$ in $G$ such that  $$(a_m \cdot a_n)_e    \ := \   ( d(a_m,e) + d(a_n,e) - d(a_m,a_n))/2  \ \to \   \infty$$ as $m,n \to \infty$, where two such sequences  $(a_n)$ and $(b_n)$ are equivalent when  $(a_m \cdot b_n)_e \to  \infty$ as $m,n \to \infty$.    It is independent of the choice of finite generating set defining $d$ and of the choice of basepoint.   See \cite{BH1} and  \cite{KB} for surveys.  

When $H$ is an infinite hyperbolic subgroup of $G$, one can seek to induce  a  map  $\partial H \to  \partial G$ from the inclusion map.  In \cite{CTnormal} and \cite{Mitra2} Mitra (or Mj, as he is now known) asks whether this is always well--defined, the concern being that $\partial H$ is defined via the word metric on $H$ and  $\partial G$ via that on $G$, and these may differ.  He cites  Bonahon \cite{Bonahon1} for similar questions and Bonahon \cite{Bonahon2}, Floyd \cite{Floyd} and Minsky \cite{Minsky} for  related  work on Kleinian groups.  The question is also raised by I.~Kapovich \& Benakli   \cite{KB} and appears in the problem lists \cite{Bestvina} and \cite{Kapovich}.   When the map exists, it is known as the Cannon--Thurston map.

The Cannon--Thurston map exists for many families of examples.   
The most straightforward is when $H$ is quasi--convex (that is, undistorted).  
Cannon \& Thurston~\cite{Cannon-Thurston} gave the first distorted example: they showed the map exists for $\pi_1 S \hookrightarrow \pi_1 M$ where $M$  is a  closed hyperbolic 3--manifold fibering over the circle with fiber a hyperbolic surface $S$ (and, strikingly, the  Cannon--Thurston map is  a group--equivariant space--filling Peano curve $S^1  \onto  S^2$).      
Mitra  generalized this widely.
He showed the Cannon--Thurston map exists when $H$ is an infinite normal 
subgroup of a hyperbolic group $G$ \cite{CTnormal}, and he developed a theory of 
ending laminations for this context (inspired by \cite{Cannon-Thurston}) to 
describe it \cite{Mitra3}.
He also showed the Cannon--Thurston map exists when $H$ is one of the 
infinite vertex-- or edge--groups of a  finite graph of groups $G$ in 
which $G$ and all of the vertex-- and edge--groups  are hyperbolic, 
and all the defining monomorphisms from edge--groups to vertex--groups are quasi--isometric embeddings~\cite{CTtrees}.

Recently, Mj established that Cannon--Thurston maps exist for surface
Kleinian groups \cite{Mj2} (answering a question of Cannon \& Thurston
from \cite{Cannon-Thurston} and Question~14 from Thurston's celebrated
1982 Bulletin AMS paper \cite{Thurston}) and then  for arbitrary
Kleinian groups \cite{Mj} (proving a conjecture of McMullen from
\cite{McMullen}).   Mitra's question can be viewed as asking whether
the natural generalization of these results in the setting of
Geometric Group Theory holds.  We answer it negatively: 

\begin{thm} \label{main}
There are positive words $C$, $C_1$, $C_2$ on $c_1, c_2$ and  $D_1, D_2, D_{11}, D_{12}, D_{21}, D_{22}$ on $d_1,  d_2$ so that 
\[G \ = \ \left\langle \    a, \, b, \, c_1, c_2, \, d_1,  d_2  \     \left| \left. \, 
\parbox{72mm}{ 
$\begin{array}{rlrl}
  a^{-1}b^{-1}ab \!\!\!\! &  =   C, &    b^{-1}c_ib\!\!\!\!& = C_i,     \\  
  (ab)^{-1}d_j(ab)\!\!\!\!&=D_j,  &    c_i^{-1}d_jc_i\!\!\!\! & =D_{ij},  \ \     1 \leq  i,j \leq 2   
\end{array}$} \,  \right. \right.  \right\rangle\]
is hyperbolic, the subgroup $$H \ = \  \langle b,d_1, d_2 \rangle$$ is free of rank $3$, and  there is no Cannon--Thurston map $\partial{H}\to\partial{G}$.
\end{thm}

 A \emph{positive} word is one in which inverses of generators never appear. 

At the expense of complicating the construction, $G$ can be made \CAT$(-1)$, as we will outline in Remark~\ref{CAT(-1)}.   

That  $H$ is free is not pertinent to the nonexistence of the Cannon--Thurston map.   Theorem~\ref{main}  is the starting point 
for a proof by Matsuda and Oguni~\cite{MO} that for every non--elementary hyperbolic group there is an embedding in some other hyperbolic group for which there is no Cannon--Thurston map. Implications of Theorem~\ref{main}  have also been explored by  Gerasimov and Potyagailo in a recent paper \cite{GP} on convergence actions. 

Given that Cannon--Thurston maps do not always exist, the question arises as to what bearing subgroup distortion has.   Heavy distortion appears to be no obstacle to the map's existence: we showed in \cite{BR1} that Cannon--Thurston maps exist for highly distorted free subgroups of hyperbolic hydra groups;  these examples exhibit the maximum distortion known among  hyperbolic subgroups of hyperbolic groups. 
As for small distortion, if a subgroup of a hyperbolic group is  subexponentially distorted, then the subgroup is quasi--convex by Proposition~2.6 of  \cite{Ksubexp} and so the Cannon--Thurston Map exists.  The natural open question, then, (which Ilya Kapovich asked us)  is whether there is an exponentially distorted hyperbolic subgroup of a hyperbolic group for which the Cannon--Thurston map does not exist.\footnote{An earlier version of this article claimed that the subgroup $G_{bcd}\leq G$ (defined before Lemma~\ref{properties}) is such an example.  Although $\partial G_{bcd}\to\partial G$ is not well-defined, we realized that the distortion is at least doubly exponential, so Kapovich's question remains open.} It will be apparent from our proof of Theorem~\ref{main} that the subgroup $H \leq G$   is at least doubly--exponentially distorted.

\begin{Acknowledgements}  We thank Ilya~Kapovich and Hamish~Short for comments and especially for Lemma~\ref{Strebel}, which replaces a more ad hoc argument in an earlier version of this article.  We are also grateful to an anonymous referee for a careful reading, improvements to our exposition, and insights on the background to Mitra's problem.  
\end{Acknowledgements}

\section{Proof of the theorem} \label{construction}
 
Denote the free group   on a set $S$ by $F(S)$.  If $S=\{s_1,\ldots,s_n\}$,  write $F(S)=F(s_1,\ldots,s_n)$.  If $F$ is a group and $X\subseteq F$ a subset so that the natural map $F(X)\to F$ is an isomorphism, then $X$ is called
a \emph{free basis} for $F$ and $F$ is said to be a \emph{free group of rank} $\textrm{card}(X)$.

We begin by showing that when  $C$, $C_i$, $D_j$ and $D_{ij}$ are chosen suitably, the group $G$ of Theorem~\ref{main} is hyperbolic. 

A finite presentation for a group satisfies the $C'(\lambda)$ small--cancellation condition  when,  after cyclically reducing all the defining relations,   the set  $S$ of all their cyclic permutations and those of their inverses,  has the property that  every common prefix between two distinct $r_1, r_2 \in S$ has length less than $\lambda$ times the lengths of each of $r_1$ and $r_2$ \cite[page~240]{LS1}. 

Following Rips~\cite{Rips}, we take
\begin{align*}
C& \ = \ c_1c_2c_1c_2^2c_1c_2^3\cdots c_1c_2^r,\\
C_i& \ = \ c_1c_2^{ri+1}c_1c_2^{ri+2}c_1c_2^{ri+3}\cdots c_1c_2^{ri+r},\\
D_j& \ = \ d_1d_2^{rj+1}d_1d_2^{rj+2}d_1d_2^{rj+3}\cdots d_1d_2^{rj+r},\\
D_{ij}& \ =  \ d_1d_2^{r(il+j)+1}d_1d_2^{r(il+j)+2}d_1d_2^{r(il+j)+3}\cdots d_1d_2^{r(il+j)+r},
\end{align*}
where $r$ is sufficiently large that the presentation for $G$ of Theorem~\ref{main} satisfies the $C' (1/6)$ condition, and so $G$ is hyperbolic.    
(All $C'(1/6)$ groups admit linear isoperimetric functions and so are hyperbolic \cite{Gersten9}.) 

Next we analyze the construction of $G$ to show (via Lemmas~\ref{Britton} and \ref{properties}  (iv), (v)) that  $H$ is free of rank $3$ for such $C$, $C_i$, $D_j$ and $D_{ij}$.  
Specifically, we will view $G$ as being built from the free group $F(d_1,d_2)$ 
by  \emph{HNN-extensions}.

Suppose $B\leq A$ are groups and $\phi:B\to A$ is any injective 
homomorphism (not necessarily the subgroup inclusion map).    The HNN-extension $A_{\ast \phi}$ of $A$ with \emph{defining homomorphism} $\phi$ and \emph{stable letter} $t$    is 
the group presented by
$$A_{\ast \phi} \ := \ \langle \, A,t \mid t^{-1}bt=\phi(b) \textup{ for all } b\in B \, \rangle,$$
where $t$ is a new generator.  (We may instead present  $A_{\ast\phi}$ by only including the relations $t^{-1}bt=\phi(b)$ for $b$ in some particular generating set for $B$.) The groups $B$ and $\phi(B)$ are called $\emph{associated subgroups}$ of the HNN-extension. 

\begin{lemma}[Britton's Lemma; e.g.\  \cite{BH1, LS1, Stillwell}] \label{Britton}
Suppose   a non-empty word $w$ on the alphabet $\{A\smallsetminus\{e\}\}\sqcup\{t,t^{-1}\}$  contains no two consecutive letters from $A\smallsetminus \{e\}$ and no subword $tt^{-1}$ or $t^{-1}t$.  Then $w\neq1$ in $A_{\ast\phi}$ unless $w$ contains a subword $t^{-1}bt$ where $b\in B$ or $tct^{-1}$ where $c\in\phi(B)$.

In particular, the natural map $A\to A_{\ast \phi}$ is injective, so $A$ can be regarded as a subgroup of $A_{\ast \phi}$ (hence ``extension''),
and $t$ generates an infinite cyclic subgroup of $A_{\ast\phi}$ trivially intersecting $A$.
  \end{lemma}

 We will need to recognize when a map between free groups is injective in order to show that it gives rise to an HNN-extension.  To this end, we will want to be able to recognize free bases.
Nielsen showed that a set of words represents a free basis for a subgroup of $F(X)$ when certain small-cancellation conditions are satisfied.
\begin{prop}[Nielsen, see {\cite[pages 6--7]{LS1}}] \label{Nielsen}
A set $U$ of words on an alphabet $X$ represents a free basis for a subgroup of $F(X)$ if for every $v_1,v_2,v_3 \in U^{\pm 1}$,
\begin{itemize}
\item[N0.] $v_1\neq e $,
\item[N1.] $v_1v_2\neq e\implies |v_1v_2| \geq |v_1|,|v_2|$,
\item[N2.] $v_1v_2\neq e$ and $v_2v_3 \neq e$ $\implies|v_1v_2v_3|>|v_1|-|v_2|+|v_3|$.
\end{itemize}
\end{prop}

\begin{cor} \label{freeness}
$C,C_1,  C_2$  span a rank--$3$ free subgroup of $F(c_1, c_2)$ and
$D_1$, $D_2$, $D_{11}$, $D_{12}$, $D_{21}$, $D_{22}$ span  a rank--$6$ free subgroup of $F(d_1, d_2)$.  
\end{cor}

(Indeed, N0--N2  are satisfied if $U$ satisfies the $C' (1/2)$ property.)

Define 
\begin{align*}
G_{cd} \  &:= \  \langle  \, c_1,  c_2, \, d_1,  d_2 \,  \mid \,  c_i^{-1}d_jc_i=D_{ij},  \  \,    1 \leq  i,j \leq 2 \, \rangle, \\
G_{bcd} \  &:= \  \langle \, G_{cd}, b \,  \mid \,   b^{-1} c_i b = C_i, \  \,     1 \leq  i  \leq 2 \,  \rangle.
\end{align*}

\begin{lemma}  \label{properties}
The groups defined above have the following properties.
\begin{enumerate}
\item[(i)]  $F(d_1,d_2)$ is a subgroup of $G_{cd}$.
\item[(ii)]  $F(c_1,c_2)$ is also a subgroup of $G_{cd}$ and $F(c_1,c_2)\cap F(d_1,d_2)=\{1\}$.
\item[(iii)]  $G_{bcd}$ is an HNN-extension of $G_{cd}$ with stable letter $b$ and defining homomorphism $\phi:F(c_1,c_2)\to G_{bcd}$ mapping $c_i \mapsto C_i$. 
\item[(iv)] $H = \langle b,d_1,  d_2 \rangle \leq G_{bcd}$ is free of rank $3$.
\item[(v)] $G$ of Theorem~\ref{main} is an HNN-extension of $G_{bcd}$ with stable letter $a$:
\begin{align*}
G  \  & = \  \langle \, G_{bcd}, a \,  \mid \,  a^{-1}ba=bC^{-1}, \  a^{-1}d_ja=bD_jb^{-1},  \  \,     1 \leq  j  \leq 2 \,  \rangle.
\end{align*}
\end{enumerate}
\end{lemma}

\begin{proof}
(i) By Corollary~\ref{freeness}, the map $\phi_1:F(d_1,d_2)\to F(d_1,d_2)$ given by $d_j\mapsto D_{1j}$ is injective.
So $\phi_1$ defines an HNN-extension of $F(d_1,d_2)$.  Calling the stable letter $c_1$, this HNN-extension has  presentation 
$$G_{c_1 d} \  := \  \langle  \, c_1, \, d_1,  d_2 \,  \mid \,  c_1^{-1}d_jc_1=D_{1j},  \  \,    1 \leq  j \leq 2 \, \rangle.$$
By Britton's Lemma, $F(d_1,d_2)\leq G_{c_1 d}$.
Similarly, $G_{cd}$ is an HNN-extension of $G_{c_1d}$ with stable letter $c_2$ and defining homomorphism
$\phi_2:F(d_1,d_2)\to G_{c_1d}$ given by $d_j\mapsto D_{2j}$.
Note that $\phi_2$ has image contained in $F(d_1,d_2)\leq G_{c_1d}$. 
Again, $\phi_2$ is injective by Corollary~\ref{freeness}.  So $F(d_1,d_2)\leq G_{c_1 d}\leq G_{cd}$ by Britton's Lemma. 

(ii)
To show that $\langle c_1,c_2\rangle$ is a free subgroup  $F(c_1,c_2)$ of $G_{cd}$ trivially intersecting $F(d_1,d_2)$, we  prove the following claim.  For any $n\geq1$, any non-identity $d\in F(d_1,d_2)$,  and any integers $r_0,\ldots,r_{n+1},s_1,\ldots,s_n$:
$$(dc_1^{r_0})c_2^{s_1}c_1^{r_1}c_2^{s_2}\cdots c_1^{r_n}c_2^{s_n}c_1^{r_{n+1}} \ \ne \  1\textrm{ in }G_{cd}$$
whenever $r_i,s_i\neq0$ for all $1\leq i\leq n$.
This, in turn, follows from Britton's Lemma applied to the extension $G_{cd}$ (which has stable letter $c_2$) once we show $dc_1^{r_0}$ and $c_1^{r_i}$ are not in the associated subgroups of $G_{cd}$.
As these associated subgroups are $F(d_1,d_2)$ and $F(D_{21},D_{22})\leq F(d_1,d_2)$, 
the observation that $\langle c_1\rangle$ is an infinite cyclic subgroup of $G_{c_1 d}$ trivially intersecting $F(d_1,d_2)$ by Britton's Lemma completes the proof.

(iii) By Corollary~\ref{freeness}, $\{C_1,C_2\}$ is a free basis of a subgroup of $F(c_1,c_2)\leq G_{bcd}$.  So the defining homomorphism $\phi:F(c_1,c_2)\to G_{bcd}$, $c_i\mapsto C_i$ is injective.

(iv)  By  Britton's Lemma   applied to the HNN-extension $G_{bcd}$ of $G_{cd}$, any freely reduced word $w$ on $b,d_1, d_2$ representing the identity would contain a subword $ b^{\pm 1} u b^{\mp 1}$ where $u$ is a  nonempty  reduced word on $d_1, d_2$  representing an  element  
 of the associated subgroup $F(c_1,c_2)$ or of the associated subgroup $\phi(F(c_1,c_2))\leq F(c_1,c_2)$.
By (ii), this is impossible.  So $H:=\langle b,d_1,d_2\rangle$ is free of rank 3.

 (v) The given presentation for $G$ arises from that in Theorem~\ref{main} by rewriting the defining relations involving $a$.
We must show that $\langle bC^{-1},bD_1b^{-1},bD_2b^{-1}\rangle\leq G_{bcd}$ is free of rank 3.  It suffices to show the same of the conjugate subgroup $\langle b^{-1}C,D_1,D_2\rangle\leq G_{bcd}$.  We do this  by  proving that if  $i_1,\ldots i_{r-1} \neq 0$ and $W_1,...,W_r$ are nontrivial elements of the rank--2 free group
$F(D_1,D_2) \leq F(d_1,d_2) \leq G_{bcd}$, then
\[
w \ : = \ (b^{-1}C)^{i_0} W_1 (b^{-1}C)^{i_1}\cdots W_r (b^{-1}C)^{i_r} 
\]
does not represent the identity in $G_{bcd}$.  This is achieved by writing $w$ in a way so that Britton's Lemma applies. 
 
The relations $b^{-1}c_ib=C_i$ imply that $(b^{-1}C)^{i_k}\in\langle c_1,c_2\rangle \, b^{-i_k} \, \langle c_1,c_2\rangle$, so:

$$w \ \in \ \langle c_1,c_2\rangle \, b^{-i_0} \, \langle c_1,c_2\rangle \, W_1 \, \langle c_1,c_2\rangle \, b^{-i_1} \, \langle c_1,c_2\rangle  \,    \cdots \, W_r \, \langle c_1,c_2\rangle \, b^{-i_r} \, \langle c_1,c_2\rangle.$$

If $b^{\pm 1}$ does not appear in  $w$, then $r=1,i_0=i_1=0$, and $w=W_1$ does not represent the identity in $G_{bcd}$.   So we may assume $b$ appears. To apply Britton's Lemma, we must show that $w$ has no subword $b^{\pm 1} Y b^{\mp 1}$  where $Y$ is a word on $c_1, c_2, d_1, d_2$ representing an element of $F(c_1,c_2)$.   This is so because $F(c_1,c_2)\cap F(d_1,d_2)=\{1\}$ by (ii) and $W_k \in F(d_1,d_2) \leq G_{bcd}$ does not represent the identity. 
\end{proof}
 
We will use the  following lemma of Mitra to show the absence of a Cannon--Thurston map  $\partial H \to \partial G$.   We  give our own account of this lemma in \cite{BR1}. 

\begin{lemma}[Mitra  \cite{CTnormal, CTtrees}] \label{Mitra's lemma}
Suppose $H$ is a hyperbolic subgroup of a hyperbolic group  $G$ and $X_H$ and $X_G$ are their Cayley graphs with respect to finite generating sets where  that for $H$  is a subset of that for $G$.  (So $X_H$ is a subgraph of $X_G$.)    Let $M(N)$ be the infimal number such that if $\lambda$ is a geodesic in $X_H$ outside the ball of radius $N$ about $e$ in $X_H$, then every geodesic in $X_G$ connecting the end--points of $\lambda$ lies outside  the ball of radius $M(N)$ about $e$ in $X_H$.   The Cannon--Thurston map $\partial H \to \partial G$ exists if and only if $M(N) \to \infty$ as $N \to \infty$.
 \end{lemma}

We will apply this to $G$ and $H$  of Theorem~\ref{main}, using the generating sets  $a, \, b, \, c_1,   c_2, \, d_1,   d_2$ and $b, \,  d_1,  d_2$, respectively.   

The next lemma identifies some geodesics in Cayley graphs of small--cancellation groups.  We learnt it from Ilya~Kapovich and Hamish~Short.  It can be extracted from Strebel's appendix to \cite{GhysdelaHarpe} as we will explain. For a finite  presentation $\langle A \mid R \rangle$, a word $w$ on $A$ is \emph{Dehn--reduced} if every subword $\alpha$ of $w$ that is a prefix of a cyclic conjugate of some $\rho \in R^{\pm 1}$ satisfies $|\alpha| \ \leq \   |\rho|/2$, and is \emph{strongly} Dehn--reduced if $|\alpha| \ \leq \   |\rho|/6$. 

\begin{lemma}  \label{Strebel}  If $\langle A \mid R \rangle$ is a $C'(1/6)$--presentation, then every strongly Dehn--reduced word  on $A$ is geodesic.  (Indeed, it is the unique geodesic word and also the unique Dehn--reduced word for the group element it represents.)  
 \end{lemma}

\begin{proof}
Suppose $u$ and $v$ are freely reduced words which represent the same group element, and   $u$ is strongly Dehn--reduced and  $v$  is geodesic.  In his proof of Proposition~39(i) in  his appendix to \cite{GhysdelaHarpe},  Strebel explains that there is a van~Kampen diagram $\Delta$ for $uv^{-1}$ whose 2--dimensional portions are \emph{ladder--like} disc--diagrams. (See the figure within Theorem~35.)

Suppose there is a 2--cell in $\Delta$ and that $\rho$ is the defining relation one reads around its boundary.   That 2--cell's boundary cycle is assembled from   four paths: two run along the boundaries of adjacent $2$--cells and have lengths less than $|\rho|/6$ (by the $C'(1/6)$  condition); one runs along $u$ and has length at most  $|\rho|/{6}$ by the strongly Dehn--reduced condition; but then the final path, which runs along $v$,  has length more than  $|\rho|/{2}$ contrary to $v$ being a geodesic word.   (Indeed, if we only required $v$ to be Dehn--reduced we would get the same contradiction.) So $\Delta$ has no $2$--cells and $u=v$ as words.
\end{proof}

\begin{proof}[Proof of Theorem~\ref{main}]
Recall that
\[G \ = \ \left\langle \    a, \, b, \, c_1, c_2, \, d_1,  d_2  \     \left| \left. \, 
\parbox{72mm}{ 
$\begin{array}{rlrl}
  a^{-1}b^{-1}ab \!\!\!\! &  =   C, &    b^{-1}c_ib\!\!\!\!& = C_i,     \\  
  (ab)^{-1}d_j(ab)\!\!\!\!&=D_j,  &    c_i^{-1}d_jc_i\!\!\!\! & =D_{ij},  \ \     1 \leq  i,j \leq 2   
\end{array}$} \,  \right. \right.  \right\rangle\]
where 
\begin{align*}
C& \ = \ c_1c_2c_1c_2^2c_1c_2^3\cdots c_1c_2^r,\\
C_i& \ = \ c_1c_2^{ri+1}c_1c_2^{ri+2}c_1c_2^{ri+3}\cdots c_1c_2^{ri+r},\\
D_j& \ = \ d_1d_2^{rj+1}d_1d_2^{rj+2}d_1d_2^{rj+3}\cdots d_1d_2^{rj+r},\\
D_{ij}& \ =  \ d_1d_2^{r(il+j)+1}d_1d_2^{r(il+j)+2}d_1d_2^{r(il+j)+3}\cdots d_1d_2^{r(il+j)+r}.
\end{align*}
We must show that for sufficiently large $r$, 
$G$ is hyperbolic, $H=\langle b,d_1,d_2\rangle$ is free of rank 3, and there is no Cannon--Thurston map $\partial H\to\partial G$. 

As we observed at the start of this section, $G$ can be made hyperbolic by choosing $r$ large enough to make $G$ satisfy $C'(1/6)$.  Britton's Lemma and Lemma \ref{properties} (iv), (v) together show that $H$ is a rank 3 free subgroup for the same $r$.  We may assume $r>17$.  It remains to show the Cannon--Thurston map does not exist.

 The longest subword of $b^{-n}a^{-n}d_1 a^n b^n$ that is a prefix of a cyclic conjugate of a defining relation or the inverse of a defining relation is $a^{-1}d_1a$. Since $r>17$, the length of $a^{-1}d_1a$ is a small fraction (less than $1/6$) of the length  of the shortest of the relators.  So $b^{-n}a^{-n}d_1 a^n b^n$ is \emph{strongly} Dehn--reduced.  So, by Lemma~\ref{Strebel}, the path $\gamma_n$ it labels, passing  through the identity $e$ as shown in Figure~\ref{Figure}, is geodesic in  the Cayley graph of $G$.

 \begin{figure}[ht]
\psfrag{a}{$a^n$}
\psfrag{b}{$b^n$}
\psfrag{u}{$u$}
\psfrag{w}{$ $}
\psfrag{e}{$e$}
\psfrag{g}{$\gamma_n$}
\psfrag{l}{$\lambda_n$}
\psfrag{d}{$d_1$}
 \centerline{\includegraphics{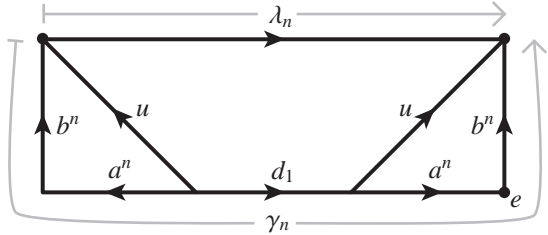}} \caption{Paths in the Cayley graph of $G$ illustrating our proof of Theorem~\ref{main}.} \label{Figure}
\end{figure}

We now wish to express $b^{-n}a^{-n}d_1 a^n b^n$ as a word in $d_1,d_2$.  To begin, we prove 
by induction on $n$ that 
\begin{equation}\label{proof abn}
ab^n \ = \ b^{n-1}ab\phi(C)\cdots\phi^{n-2}(C)\phi^{n-1}(C)
\end{equation}
in $G$, where $\phi:F(c_1,c_2)\to G_{bcd}$, $c_i\mapsto C_i$ is the defining homomorphism of the HNN-extension $G_{bcd}$ with stable letter $b$---see Lemma~\ref{properties}~(iii).
The base case $n=1$ is the equation $ab=ab$.  The induction step follows from the relation $a^{-1}b^{-1}ab=C$
(which rearranges to $ab=baC$):
\begin{eqnarray*}
ab^{n+1} \ = \ (ab)b^n \ = \ (baC)b^n
  &=& b(ab^n)(b^{-n}Cb^n)
\\   &=& b(ab^n)\phi^n(C)
\\ &=& b(b^{n-1}ab\phi(C)\cdots\phi^{n-2}(C))\phi^{n-1}(C))\phi^n(C),\end{eqnarray*} where the last equality uses the induction hypothesis.
Left-multiplying (\ref{proof abn})  by $a^{n-1}$ yields:
\begin{equation}\label{proof anbn}
a^nb^n \ = \ (a^{n-1}b^{n-1})ab\phi(C)\cdots\phi^{n-2}(C))\phi^{n-1}(C)
\end{equation}
Another induction then shows that $a^nb^n$ can be written as a positive word $u$ in the alphabet $\{ab,c_1,c_2\}$.
So $b^{-n} a^{-n} d_1 a^n b^n \ = \ u^{-1} d_1 u$  in $G$, which  equals a positive word on  $d_1,d_2$ since  $(ab)^{-1}d_j(ab) =D_j$ and $c_i^{-1}d_jc_i =D_{ij}$ in $G$.     

So the endpoints of $\gamma_n$ are in $H$, and the geodesic $\lambda_n$ joining them in the Cayley graph of  $H$ (which is a  tree) is labelled by a word on $d_1, d_2$. 
The distance (along the path labelled $b^n$) from $e$ to $\lambda_n$ in $H$ is $n$. 

As the distance from $\gamma_n$ to $e$ in the Cayley graph of $G$ is zero and the distance from $\lambda_n$ to $e$  in the Cayley graph of $H$ is $n$,  
there is no Cannon--Thurston map $\partial H \to \partial G$ by  Lemma~\ref{Mitra's lemma}. 
\end{proof}

\section{Remarks}

\begin{remark} 
The inclusion $H \hookrightarrow G$ factors through $G_{bcd}$, which is also hyperbolic as its presentation is also $C'(1/6)$.  So Theorem \ref{main} implies the absence of at least one Cannon--Thurston map $\partial H\to\partial G_{bcd}$ or $\partial G_{bcd}\to\partial G$.
In fact, more elaborate versions of the argument given above establish that  both fail to exist.
As an HNN--extension is an example of a graph of groups, the latter example also shows that the quasi--isometric embedding hypothesis in Mitra's theorem from \cite{CTtrees} is necessary.
  \end{remark}

\begin{remark} \label{CAT(-1)}
With a similar construction, one can obtain a $\textup{CAT}(-1)$ group $G$ with a free subgroup $H$ with no Cannon-Thurston map.
Wise's modification in \cite{wise1998} of the Rips construction \cite{Rips} is used in \cite{BBD} to construct $\textup{CAT}(-1)$ groups.
Each relator is realized on the boundary of the unions of $n=5$ congruent right--angled regular hyperbolic pentagons, arranged as row houses atop a geodesic segment.
Each edge of the boundary corresponds to a generator.  The vertices of the boundary are either right angles or straight angles,
but the base geodesic gives $n-1$ consecutive straight angles, bounding a segment of length $n-2$.
Wise shows that the Gromov link condition is satisfied when this straight segment is a freely reduced word and when the length--$(2n+4)$ remainder of the boundary is obtained from the \emph{Wise word}:
\[c_1(c_1c_2c_1c_3\cdots c_1c_r)c_2(c_2c_3c_2c_4\cdots c_2c_r)c_3(c_3c_4\cdots c_3c_r)\cdots c_{r-1}(c_{r-1}c_r)c_r\]
by chopping it into consecutive length $2n+4$ segments (one for each defining relator).  
The argument works just as well for any $n$, so we take $n=7$ and fit the $(ab)^{-1}d_j(ab)$, $a^{-1}b^{-1}ab$, $b^{-1}c_ib$, and $c_i^{-1}d_jc_i$ portions of our relators along the straight segment.
We form one Wise word of $c$'s and one of $d$'s.  To get sufficiently many length--$18$ subwords of the Wise words, we increase the number of $c_i$ and $d_j$ in the generating set for $G$.
Then $H=\langle b,d_1,d_2,\ldots\rangle$ is a free subgroup of the $\textup{CAT}(-1)$ group $G$ by the same argument as before.
\end{remark}

\begin{remark} \label{height}
$H$ has \emph{infinite height} in $G$.
That is, for all $n$, there exist 
$g_1, \ldots, g_n \in G$ such that    $\bigcap_{i=1}^n {g_i}^{-1} H g_i$ 
is infinite and $H g_i \neq  H g_j$ for all $i \neq j$.
Specifically, take $g_i =c_1^i$.  
Then, if $\phi_{1} : F(d_1,d_2) \to F(d_1, d_2)$ is the 
map $d_j \mapsto D_{1j}$ for $j=1,2$, then $\phi_{1}^n(F(d_1, d_2))$ is an 
infinite subgroup inside $g_i^{-1} H g_i$ for $1 \leq i \leq n$, and  $H 
g_i \neq  H g_j$ for all $i \neq j$ since $c_1^k \in H$ only for $k=0$ by    
Lemma~\ref{properties}.
Likewise, $G_{bcd}$ has infinite height in $G$: 
instead of taking $g_i=c_1^i$, take $g_i=(ab)^i$ and apply the same 
argument as above.
So our examples do not resolve the question attributed 
to Swarup in \cite{Mitra2}: if $H$ is a finitely presented subgroup of a 
hyperbolic group $G$ and $H$ has \emph{finite height} in $G$, 
is $H$ quasiconvex in $G$?
\end{remark}

\bibliographystyle{plain}
\bibliography{CT_counter}

\end{document}